\documentclass{amsart}
\usepackage[utf8]{inputenc}
\usepackage{amssymb}
\usepackage{amsthm}
\usepackage[colorlinks]{hyperref}
\usepackage{enumitem}

\numberwithin{equation}{section}

\newtheorem{thm}{Theorem}
\newtheorem{lemma}[thm]{Lemma}

\newtheorem{rmk}[thm]{Remark}
\newtheorem{prop}[thm]{Proposition}
\theoremstyle{definition}

\newtheorem{example}[thm]{Example}

\begin{document}
\title{An elementary example of Sard's Theorem sharpness}

\date{\today\ (\the\time)}

\author{Juan Ferrera}
\address{IMI, Departamento de An{\'a}lisis Matem{\'a}tico y Matem{\'a}tica Aplicada,
Facultad Ciencias Matem{\'a}ticas, Universidad Complutense, 28040, Madrid, Spain}
\email{ferrera@mat.ucm.es}


\begin{abstract}
In this note we define a $C^1$ function $F:[0,M]^2\to [0,2]$ that satisfies that its set of critical values 
has positive measure. This function provides an  example, easier than those that usually appear in the
literature, of how the order of differentiability required in 
Sard's Theorem cannot be improved
\end{abstract}

\maketitle

The classical Sard's Theorem, see \cite{M} and \cite{Sard}, asserts that a $C^{n-m+1}$ function $F:\mathbb{R}^n\to \mathbb{R}^m$
satisfies that the set of its critical values has measure $0$. A classical example of Whitney (see \cite{W})
shows that the result is sharp within the classes of functions $C^k$. Specifically he built  a 
$C^1$ function
$F:\mathbb{R}^2\to \mathbb{R}$ such that all its values are critical. In this note we present an 
easier example of that fact. This example has another interesting property, namely 
the function $F$ has $\frac{1}{2}$-H\"older continuous derivatives, that is $F$ is $C^{1,\frac{1}{2}}$.
This is a well known fact, indeed Norton (see \cite{N}) provides examples of functions $C^{1,s}$ whose
set of critical values contains and interval, for every $s<1$, anyway we consider that the example that we provide
is much easier. Moreover, this example is self contained in the sense that it only requires the knowledge of the Cantor set
and elementary calculus; this is important since the examples that usually appear require deeper results as Whitney's 
Extension Theorem for instance.

It is interesting to compare these examples with 
the improvement of Sard's Theorem due
to Bates, see \cite{B}, that affirms that in order
to guarantee that the set of critical values is a null set we only have to require that the function $F$ be $C^{n-m,1}$.

Before to present the example,
in order to fix notation, we define the Ternary Cantor Set $C$ and we show two of its properties. 
We choose a way to introduce $C$ that will be useful while defining the goal function.   
$$
C=[0,1]\smallsetminus \bigcup_{n=1}^{\infty}I_n, \qquad I_n=\bigg( \bigcup_{k=1}^{2^{n-1}}I_n^k\bigg)
$$
where the $2^{n-1}$ intervals $I_n^k$, of length $\frac{1}{3^n}$, are centered in the middle points of the connected
components of $C\smallsetminus  \big( I_1\cup \dots \cup I_{n-1}\big)$.
We start remembering a well known fact

\begin{prop}\label{P1}
The Ternary Cantor set $C$ satisfies that $C+C=[0,2]$.
\end{prop}

Indeed, it was Steinhaus who first proved this result in 1917. However, for the sake
of self containedness, we present a short proof due to Shallit, \cite{Sh} that we found in
\cite{ART}.
\begin{proof}
Given $u\in [0,2]$, we consider the basis three expansion of $u/2$,
$$
\frac{u}{2}=\sum_{n=1}^{\infty}\frac{\varepsilon_n}{3^n} \qquad \varepsilon_n\in \{ 0,1,2\} .
$$
If $\varepsilon_n=0$, we define $\alpha_n=\beta_n=0$, 
if $\varepsilon_n=1$, then $\alpha_n=2$ and $\beta_n=0$, finally if
$\varepsilon_n=2$, then $\alpha_n=\beta_n=2$.  We have that $\alpha_n+\beta_n=2\varepsilon_n$.
The numbers
$$
x=\sum_{n=1}^{\infty}\frac{\alpha_n}{3^n}\quad \textrm{and}\quad y=\sum_{n=1}^{\infty}\frac{\beta_n}{3^n}
$$
belong to $C$ since $\alpha_n\neq 1\neq \beta_n$ for every $n$. It is immediate that
$x+y=2\frac{u}{2}=u$. This proves that $[0,2]\subset C+C$ which is enough since the other inclusion is trivial.
\end{proof}

Another immediate property which is a consequence of the measure zero of the Cantor set
is the following one:

\begin{prop}\label{P2}
Every $x\in C$ satisfies
$$
x=\sum \mathcal{L}(I_n^k)
$$
where the sum ranges over all the intervals satisfying that $\sup I_n^k\leq x$.
\end{prop}

We proceed to define a Cantor type set. Let 
$$
M=\sum_{n=1}^{\infty}\frac{2^{n-1}}{\big( 3^{2/3}\big) ^n}=
\frac{1}{3^{2/3}-2}, 
$$
we define $A\subset [0,M]$ as the Cantor set, but taking instead of the interval $I_n^k$, 
new intervals
$J_n^k$ of length $\big( 3^{2/3}\big) ^{-n}$. Observe that the intervals $J_n^k$ are
pairwise disjoint,
hence 
$$
A=[0,M]\smallsetminus J ,\quad \textrm{where}\ J=\bigcup_{n=1}^{\infty}J_n
\quad \textrm{and}\ 
J_n=\bigcup_{k=1}^{2^{n-1}}J_n^k
$$
which implies that $\mathcal{L}(J_n)=2^{n-1}\big( 3^{2/3}\big) ^{-n}$
and $\mathcal{L}(J)=M$. Hence $\mathcal{L}(A)=0$.

We define a function $g:[0,M]\to \mathbb{R}$ as $g(x)=0$ for every $x\in A$, if
$x\notin A$ then $x\in J_n^k$ for some $n$ and $k$. We define $g$ on each $J_n^k$
by
$$
g(x)=\frac{4}{\mathcal{L}(J_n^k)^{\frac{1}{2}}}dist(x,\partial J_n^k)
$$ 

\begin{lemma}\label{L3}
$g:[0,M]\to \mathbb{R}^+$ satisfies the following conditions:
\begin{enumerate}
  \item $\max \{ g(x): x\in J_n^k\} =2(\mathcal{L}(J_n^k))^{\frac{1}{2}}$
 \item $g(x)=0$ in an only if $x\in A$
  \item 
  $$
\int_{J_n^k}g(x)dx=(\mathcal{L}(J_n^k))^{\frac{3}{2}}=\frac{1}{3^n}
\quad \textrm{and}\ \int_0^Mg(x)dx=1.
$$
\item $g$ is  $\frac{1}{2}$-H\"older continuous.   
\end{enumerate}
\end{lemma}
\begin{proof}
Only the last statement is not trivial. If $x,y\in A$ there is nothing to prove.
If $x,y\in J_n^k$, $x\neq y$, then
$$
|g(x)-g(y)|\leq \frac{4}{\mathcal{L}(J_n^k)^{\frac{1}{2}}}|x-y|\leq
\frac{4}{|x-y|^{\frac{1}{2}}}|x-y|=4|x-y|^{\frac{1}{2}}.
$$
For all the other situations assume that $g(x)>g(y)$, that $x\in J_n^k$, and that
$z$ is the extreme of $J_n^k$ which lies between $x$ and $y$ (including the case $z=y$).
Then
$$
|g(x)-g(y)|\leq g(x)\leq \frac{4}{\mathcal{L}(J_n^k)^{\frac{1}{2}}}|x-z|\leq |x-z|^{\frac{1}{2}}
\leq |x-y|^{\frac{1}{2}}.
$$

\end{proof}

Next we define $f:[0,M]\to [0,1]$ as
$$
f(x)=\int_0^xg(t)dt.
$$
The following lemma summarizes the properties of $f$ that we require.

\begin{lemma}\label{L4}
The function $f:[0,M]\to [0,1]$ satisfies
\begin{enumerate}
  \item $f$ is $C^1$.
  \item $f'(x)=0$ if and only if $x\in A$.
  \item $f(A)=C$.
\end{enumerate}
\end{lemma}
\begin{proof}
The first two properties are an immediate consequence of the Fundamental Theorem of Calculus. 
For the third one, we observe first that $f$ is one to one (it is a strictly increasing function),  hence
in order to obtain the result it is enough to prove that
$$
f(J_n^k)=I_n^k \quad \textrm{for every}\ n \ \textrm{and}\ k,
$$
and this follows since $f(J_n^k)$ is an interval of length $\frac{1}{3^n}$, by Lemma \ref{L3},
whose left extreme agrees with the left extreme of $I_n^k$ by Proposition \ref{P2}
\end{proof}

We are ready to set the example which is the goal of this note, that follows immediately from
Proposition \ref{P1} and Lemma \ref{L4}.

\begin{example}
The function $F:[0,M]^2\to [0,2]$ defined by $F(x,y)=f(x)+f(y)$ is a $C^{1,\frac{1}{2}}$ function
that satisfies that every $x\in [0,2]$ is a  critical value.
\end{example}
\begin{proof}
$\nabla F(x,y)=(0,0)$ if and only if $(x,y)\in A\times A$, and $F(A\times A)=f(A)+f(A)=C+C=[0,2]$
\end{proof}

\begin{rmk}
All the arguments required by this example still hold if we define 
$J_n^k$ of length $\big( 3^{\frac{1}{1+\alpha}}\big) ^{-n}$, with
$2<3^{\frac{1}{1+\alpha}}$, which is equivalent to $\alpha <\frac{\log 3/2}{\log 2}$,
and 
$$
g(x)=\frac{4}{\mathcal{L}(J_n^k)^{1-\alpha}}dist(x,\partial J_n^k).
$$ 
Then we obtain that $f$ is $\alpha$-H\"older, and consequently $F$ is $C^{1,\alpha}$.
\end{rmk}

\end{document}